\documentclass[a4paper,10pt]{article}
\usepackage[cp1251]{inputenc}
\usepackage{amssymb, amsmath, amsthm}
\title{Growth of $p$th means of analytic and subharmonic functions in the unit disc and angular distribution of zeros}
\author{Igor Chyzhykov}

\begin{document}

\maketitle
\def\th{\theta}
\def\ve{\varepsilon}
\def\vfi{\varphi}
\def\Z{\Bbb Z}
\def\N{\Bbb N}
\def\R{\Bbb R}
\newcommand{\D}{\mathbb{D}}
\newcommand{\C}{\mathbb{C}}
\newcommand{\Bl}{\Bigl(}
\newcommand{\Br}{\Bigr)}
\newcommand{\Bm}{\Bigl|}
\newcommand{\z}{\zeta}
\newcommand{\A}{\mathcal{A}}
\newcommand{\rr}[1]{re^{i{#1}}}
\newcommand{\inp}{\int_0^{2\pi}}

\newtheorem{theorem}{Theorem}
\newtheorem{lemma}{Lemma}
\newtheorem{proposition}{Proposition}
\newtheorem{corollary}[theorem]{Corollary}

\newtheorem{atheorem}{Theorem}
\renewcommand{\theatheorem}{\Alph{atheorem}}
\newtheorem{alemma}{Lemma}[atheorem]
\renewcommand{\thealemma}{\Alph{atheorem}}

\newtheorem{rem}[theorem]{Remark}

\newtheorem{ex}[theorem]{Example}

\def\arg{\mathop{\mbox{\rm arg}}}
\def\Arg{\mathop{\mbox{Arg}}}

\def\th{\theta}
\def\vfi{\varphi}
\def\si{\sigma}
\def\Z{\mathbb Z}
\def\N{\mathbb N}
\def\R{\mathbb R}
\def\C{\mathbb C}
\def\Ga{\Gamma}
\def\ga{\gamma}
\def\de{\delta}
\def\De{\Delta}
\def\om{\omega}
\def\ve{\varepsilon}
\def\be{\beta}
\def\ity{\infty}
\def\a{\alpha}
\def\al{\alpha}
\def\la{\lambda}
\def\La{\Lambda}
\def\vfi{\varphi}
\def\si{\sigma}
\def\Z{\mathbb Z}
\def\N{\mathbb N}
\def\R{\mathbb R}
\def\C{\mathbb C}
\def\D{\mathbb D}
\def\Ga{\Gamma}
\def\ga{\gamma}
\def\de{\delta}
\def\om{\omega}
\def\ve{\varepsilon}
\def\be{\beta}
\def\ity{\infty}
\def\tg{\mathop{\mbox{tg}}}
\def\mes{\mathop{\mbox{mes}}}
\def\arctg{\mathop{\mbox{arctg}}}
\def\arcsin{\mathop{\mbox{arcsin}}}
\def\const{\mathop{\mbox{const}}}
\def\Ln{\mathop{\mbox{Ln}}}
\def\AC{\mathop{{\rm AC}}}
\def\BV{\mathop{{\rm BV}}}
\def\dis{\displaystyle}

\def\hd{H(\D)}
\newcommand{\ti}{\widetilde}
\renewcommand{\Re}{\mathop{\rm Re}}
\renewcommand{\Im}{\mathop{\rm Im}}

\def\supp{\mathop{\mbox{supp}}}
\def\tg{\mathop{\mbox{tg}}}
\def\mes{\mathop{\mbox{mes}}}
\def\arctg{\mathop{\mbox{arctg}}}
\def\arcsin{\mathop{\mbox{arcsin}}}
\def\const{\mathop{\mbox{const}}}
\def\Ln{\mathop{\mbox{Ln}}}
\def\SH{\mathop{\mbox{SH}}}
\def\dis{\displaystyle}
\def\noi{\noindent}

\begin{center}
  \it In memory of Professor Anatolii Grishin
\end{center}

\begin{abstract}
Answering a question of A.Zygmund in  \cite{MR} G.MacLane and L.Rubel described boundedness of $L_2$-norm w.r.t. the argument of $\log |B|$, where $B$ is a Blaschke product. We generalize their results in several directions. We describe growth of $p$th means, $p\in(1, \infty)$,  of subharmonic functions bounded from above in the unit disc. Necessary and sufficient conditions are formulated in terms of the complete measure (of a subharmonic function) in the sense of A.Grishin. We also prove sharp estimates of the growth of $p$th means of analytic and subharmonic functions of finite order in the unit disc.
\end{abstract}

\section{Introduction and main results}
\subsection{Some results on growth and angular distribution of zeros of Blaschke products}
In the present paper we investigate an interplay between zero distribution and growth of analytic functions in the  unit disc $\mathbb{D}=\{z\in \mathbb{C}: |z|<1\}$. Especially we are interested in growth of logarithmic means.

Given a sequence $(a_n)$ in $\D$ such that  $\sum_n(1-|a_n|)<\infty$, we   consider the Blaschke product
\begin{equation}\label{e:bla}
  B(z) = \prod\limits_{n = 1}^\infty   \frac{\overline {a_n} {(a_n  -
z)}}{|a_n|(1 - z\overline {a_n }) }. 
\end{equation}

It was A. Zygmund (see \cite{MR}) who asked to describe those sequences $(a_n)$ in $\D$ that $$I(r)=\frac1{2\pi}
\int_{-\pi}^\pi (\log|B(re^{i\theta})|)^2 \, d\theta $$ is bounded.
In \cite{MR} G.\ Maclane and L.\ Rubel answered this question  using Fourier series method.

\begin{atheorem}[{\cite[Theorem 1]{MR}}] \label{t:MR1} A necessary and sufficient condition that $I(r)$  be bounded is that $J(r)$ be bounded, where
$$ J(r)=\sum_{k=1}^\infty \frac 1{k^2}\Bigl| (r^k-r^{-k})\sum_{|a_n|\le r} \bar a_n^k  + r^k \sum_{|a_n|>r} (\bar a_n ^k - a_n^{-k})\Bigr|^2.$$
\end{atheorem}

Since it was difficult to check boundedness of $J(r)$ they gave also the following sufficient condition.

Let $n(r,B)$ be the number of zeros in the closed disc  $\overline{D}(0,r)$; here and in what follows $D(z,r)=\{ \zeta\in \mathbb{C}: |\zeta -z|<r\}$.

\begin{atheorem}[{\cite{MR}}] \label{t:MR2}
If
\begin{equation}\label{e:nr_bound2}
n(r, B)=O((1-r)^{-\frac12}),  \quad  r\in (0,1),
\end{equation}
 then $I(r)$ is bounded.
\end{atheorem}
They also noted that \eqref{e:nr_bound2} is equivalent  to the condition $$\sum_{|a_n|>r} (1-|a_n|)=O(\sqrt{1-r}).$$
MacLane and Rubel also proved that \eqref{e:nr_bound2} becomes necessary if all zeros lie on a finitely many rays emanating the origin, but it is not the case in general. After that C.\ N.\ Linden (\cite[Corollary 1]{Li71}) generalized this showing that it is sufficient to require that the zero sequence  is contained in a finite number of Stolz angles with vertices on $\partial \D$. The last assertion is a consequence of the following result.

Let $${\mathcal R}(\rr\vfi, \sigma)=\Bigl\{\zeta:  r\le |\zeta|\le \frac {1+r}2 , |\arg \zeta -\vfi|\le \sigma\Bigr\}. $$

\begin{atheorem}[{\cite[Theorem 1]{Li71}}] 
If $I(r)< M$, $0<r<1$. Then
\begin{gather}
    \label{e:num_gam_est}  \# \{ a_n\in \mathcal R (\rr\vfi, \varkappa (1-r)^{\gamma})\}\le  \begin{cases}  \displaystyle \frac{C\sqrt{M} (1+\sqrt \varkappa)}{r(1-r)^{\frac 12}}, & \gamma \ge 1, \\ \displaystyle
 \frac{C\sqrt{M} (1+\sqrt \varkappa)}{r(1-r)^{1-\frac \gamma 2}}, & 0\le \gamma < 1. \end{cases}
\end{gather}
  \end{atheorem}

  Results of MacLane and Rubel show that the order of magnitude of the first estimate  \eqref{e:num_gam_est} is the best possible. Linden (\cite{Li71}) also established sharpness of the  estimate for $\gamma \in [0,1)$.

The main growth characteristic which will be studied here is  $p$th integral mean of $\log |f|$, where $f$ is analytic in $\D$.  Since our approach is natural for subharmonic functions    we introduce means for the class of subharmonic functions in $\D$. Note that $\log |f|$ is subharmonic provided that $f$ is analytic.  Characterization of zeros of analytic functions $f$ with $\log |f|\in L^p(\D)$ is obtained in~\cite{BrOC}.

For  a subharmonic function $u$ in $\D$ and   $p\ge 1$ we define
\begin{gather*}
    m_p(r,u)=\biggl( \frac 1{2\pi} \int_0^{2\pi} |u(\rr\th)|^p\, d\th\biggr)^{\frac 1p}, \quad 0<r<1, \\
    \rho_p[u]=\limsup _{r\uparrow 1}\frac{\log^+ m_p(r,u)}{-\log(1-r)}.
\end{gather*}

The growth of $m_p(r, \log|f|)$ was studied in many papers, for instance  \cite{Li_mp}, \cite{Li_mp1},  \cite{VM}, \cite{Sto83}, \cite{Gar88}, \cite{Sto89}, \cite{Chmm}, \cite{Chy_Ska}.  Nevertheless, to the best of our knowledge, only one  paper, namely \cite{VM}, contains criteria of boundedness of $p$th means when $u=\log|B|$. Unfortunately, proofs have not been  published yet.

In \cite{VM} Ya.V. Mykytyuk and Ya.V. Vasyl'kiv  introduced two auxiliary functions defined on $\partial \D$ by $(a_n)$:
$$ \psi_r(\zeta)=\sum_{r\le |a_n|<1} \frac{(1-|a_n|)^2}{|\zeta-a_n|^2},\quad \zeta \in \partial \D, r\in [0,1).$$
and $\vfi(\zeta)$, which  satisfies the relation
$$ \varphi(\zeta) \asymp \Phi(\zeta):= \# \{ a_n : |1-a_n \bar\zeta|< 2(1-|a_n|)\}, $$
i.e. the number of zeros in the Stolz angle with the vertex $\zeta$.
They established  that $\psi_0$ and $\Phi$ belong to the same classes $L^p(\partial \mathbb{D})$, $p\in [1, \infty)$, and $\psi_0 \log |\psi_0|$ and $\Phi\log |\Phi|$ belong to $L^1(\partial \mathbb{D})$, simultaneously. Moreover, for a branch of $\log B$ in $\D$ with the radial cuts $[a_k, \frac{a_k}{|a_k|})$ the following statement holds.


\begin{atheorem} [\cite{VM}] 
Let   $B$ be a Blaschke product, and $p\in (1, \infty)$.
Then:
\begin{itemize}
  \item [1)]  $m_p(r, \log B)$ is bounded on $[0,1)$ if and only if $\psi_0 \in L^p(\partial D)$.
  \item [2)] $m_1(r, \log B)$  is bounded if and only if $\psi_0 \log ^+ \psi_0 \in L^1(\partial \D)$.
\item [3)] $m_p(r, \log |B|)$ is bounded on $[0,1)$ if and only if $$
\sup_{0<r<1}  \int_{0}^{2\pi}\biggl( \int_{0}^{2\pi} \frac{1-r^2}{|\rr\theta -e^{i\vfi}|^2} \psi_r(e^{i\theta}) d\theta \biggr)^p\, d\vfi<\infty.$$
\item [4)] $ \psi_0\in L^p(\partial \D)\Rightarrow \sup_{0<r<1}m_p(r, \log|B|)<\infty.$
\item [5)] $ n(r,f)=O((1-r)^{-\frac{1}{p}})\Rightarrow \sup_{0<r<1}m_p(r, \log|B|)<\infty.$
\end{itemize}
\end{atheorem}

%

Relations between conditions on the zeros of a Blaschke product $B$
and the belongness of $\arg B(e^{i\th})$ to $L^p$ spaces $0<p\le
\infty$ were investigated by A. Rybkin (\cite{Ry}).

The following tasks arise naturally:
\begin{itemize}
  \item [i)] Describe the growth of $p$th means of $\log|f|$ where $f$ is a bounded analytic function in  $\D$, $1<p<\infty$.
  \item [ii)] Find more `explicit' conditions on zero distribution than  that of Theorems~A and D.
  \item [iii)] Extend the description on functions of finite order of  growth.
  \item [iv)] Find simple conditions that provide a prescribed growth of $m_p(r,\log|f|)$.
\end{itemize}

In the paper  we accomplish these tasks.

\subsection{Complete measure  and main results}
Our method is based on a concept of so the called \emph{complete measure} of a subharmonic function introduced by A. Grishin in the case of the half-plane (see \cite{Gr1}, \cite{FG}). As it was mentioned there, this concept allows to obtain very simple representation for a subharmonic function of finite order and defines this function up to a harmonic addend in the closure of the domain.

Let $SH^\infty $ be the class of  subharmonic functions in
$\D$ bounded from above. In  particular, $\log|f|\in SH^\infty
$ if $f\in H^\infty$, the space of bounded analytic functions in $\D$. In this case  (cf. \cite[Ch.3.7]{HK})
\begin{equation}\label{e:uzob}
  u(z)=\int_{\D} \log\frac{|z-\zeta|}{|1-z\bar\zeta|} d\mu_u(\zeta)- \frac1{2\pi}\int_{\partial \D} \frac{1-|z|^2}
  {|\zeta-z|^2}
  d\psi(\zeta) +C.
\end{equation}
where $\psi$ is a positive Borel  measure, $\mu_u$ is
 the Riesz measure of $u$ (\cite{HK}), and $\int_{\D} (1-|\zeta|) d\mu_u(\zeta)<\infty$.
The \emph{complete measure
$\lambda_u$ of $u$ in the sense of Grishin} is defined \cite{Gr1,FG} by the boundary measure  and the Riesz measure
of $u(z)$. But, since \cite{CL} $$\lim_{r\uparrow1}
\int_{\theta_1}^{\theta_2} \int_{\D} \log\frac{|\rr\th-\zeta|}{|1-\rr \th\bar\zeta|}  d\mu_u(\zeta) d\th =0, \quad -\pi\le
\th_1<\th_2\le \pi,$$
i.e. the boundary values of the first integral from (\ref{e:uzob})  do not contribute to the boundary measure,
we can define $\lambda_u$ of a Borel set
$M\subset \overline{\D}$ such that $M\cap\partial{\D}$ is measurable with respect to Lebesgue measure on $\partial \D$ by
\begin{equation}\label{e:cm}
  \lambda_u(M)= \int_{\D\cap M} (1-|\zeta|)\, d\mu_u(\zeta) + \psi(M\cap \partial \D).
\end{equation}
The measure $\lambda=\lambda_u$ has the following properties:

\begin{itemize}
  \item [(1)] $\lambda$ is finite on $\overline\D$;
  \item [(2)] $\lambda$ is non-negative;
  \item  [(3)] $\lambda$ is a zero
measure outside $\overline{\D}$; \item[(4)]
$d\lambda\Bigr|_{\partial \D}(\zeta)=d\psi(\zeta)$; \item[(5)]
$d\lambda\Bigr|_{\D}(\zeta)= \ (1-|\zeta|)\, d\mu_u(\zeta)$.
\end{itemize}

If $ B$ is a Blaschke product of
form (\ref{e:bla}), then $\lambda_{\log| B|}(M)= \sum_{a_n\in M}
(1-|a_n|)$.

Let $$\mathcal{C}(\vfi, \delta)=\{ \zeta \in \overline{\mathbb{D}}: |\zeta|\ge 1-\delta, |\arg \zeta -\varphi |
\le \pi \delta\}$$
be the Carleson box based on the arc $[e^{i(\varphi-\pi\delta)}, e^{i(\vfi+\pi\delta)}]$.

The following theorem describes the growth of integral means for $u\in SH^\infty$.
\begin{theorem} \label{t:u_bound}
Let $u\in SH^\infty$, $\gamma \in (0,1]$, $p\in (1, \infty)$. Let $\lambda$ be the complete measure of $u$. Necessary and sufficient that
\begin{equation}\label{e:u_m_p_est}
    m_p(r,u)=O((1-r)^{\gamma-1}), \quad r\uparrow 1,
\end{equation}
hold is that
\begin{equation}\label{e:u_lambda_p}
    \biggl( \inp \lambda^p(\mathcal{C}(\vfi, \de))\, d\vfi \biggr)^{\frac 1p} =O(\delta^\gamma), \quad 0<\delta<1.
\end{equation}
\end{theorem}

\begin{theorem} \label{t:f_bound}
Let $f\in H^\infty$, $\gamma \in (0,1]$, $p\in (1, \infty)$. Let $\lambda$ be the complete measure of $\log |f|$. Necessary and sufficient that
\begin{equation}\label{e:m_p_est}
    m_p(r,\log|f|)=O((1-r)^{\gamma-1}), \quad r\uparrow 1,
\end{equation}
hold is that \eqref{e:u_lambda_p}.
\end{theorem}

{\bf Remark.} It was proved in \cite{Chmm} that if $\mathop{\rm supp} \lambda \subset \partial \D$, i.e., $u$ is harmonic,   $\gamma\in (0,1)$ then \eqref{e:u_lambda_p} is equivalent to \eqref{e:u_m_p_est}.

{\bf Remark.} Though Theorems 1 and 2 look like Carleson type results we cannot  use standard tools (e.g. \cite[Chap.9]{Dur}) here, because $u$ and $\log |f|$  have logarithmic singularities.

The crucial role in the  proof of sufficiency  plays Lemma \ref{l:mes_p_est}.
In order to prove necessity of Theorems 1 and 2 we essentially use the fact  that kernels in representation  \eqref{e:uzob} preserve the sign. The method allows to spread the sufficient part of Theorems 1 and 2 to functions of finite order of growth (see Theorems \ref{t:u_finite}, \ref{t:4} below).

Under additional assumptions on  zero location of a Blaschke product (support of the Riesz measure of the Green  potential) \eqref{e:lambda_p} could be simplified.
\begin{theorem} \label{t:u_supp_stoz}
Let
\begin{equation}\label{e:u_can_int_bl}
u(z)=\int_{\D} \log\frac{|z-\zeta|}{|1-z\bar\zeta|} d\mu_u(\zeta), \quad \int_{\D} (1-|\zeta|) d\mu_u(\zeta)<\infty,
\end{equation}
$\alpha \in [0,1)$, $p\in (1, \infty)$, $\alpha+\frac 1p$.
Suppose that $\supp \mu_u$ is contained in a finite number of Stolz angles with vertices on $\partial \D$. Necessary and sufficient that
\begin{equation}\label{e:u_m_p_est}
    m_p(r,u)=O((1-r)^{-\alpha})), \quad r\uparrow 1,
\end{equation}
hold is that
\begin{equation}\label{e:nru_stolz}
n(r,u):=\mu_u( \overline{D(0,r)}=O((1-r)^{-\alpha-\frac 1p}), r\uparrow 1.
\end{equation}
\end{theorem}

Similar to the case $\alpha=0$, the growth  condition  \eqref{e:nru_stolz} is appeared to be sufficient for \eqref{e:u_m_p_est} when $u$ is of finite order (see Theorem \ref{t:u_suf_groth_al} below).

{\bf Remark.} Taking $u=\log |B|$, we obtain a generalization of MacLane and Rubel, and Linden's results mentioned in Subsection 1.1.

{\bf Remark.} If $\alpha+\frac 1p \ge 1$ correlations \eqref{e:u_m_p_est} and  \eqref{e:nru_stolz}  become trivial, see remark after Theorem \ref{t:4}.

\subsection{Growth and zero distribution of zeros of finite order}
In order to formulate results on angular distribution for unbounded analytic functions we need some growth characteristics.
The standard characteristics are the maximum modulus $M(r,f)=\max\{|f(z)|: |z|=r\}$, and the Nevanlinna characteristic (\cite{Ha}) $ T(r,f)=\frac1{2\pi}\int_0^{2\pi}\log ^+|f(re^{i\theta})| \, d\theta$, $x^+=\max\{x, 0\}$. Note that both of them  are  bounded for $f=B$. Note that the order defined by $T(r,f)$ coincides with $\rho_1[\log |f|]$.

It follows from results of C.Linden \cite{L_rep} that $M(r,f)$ does take into account the angular distribution of the zeros when it grows sufficiently fast, namely, when
the order of growth
$$\rho_M[f]=\limsup_{r\uparrow 1} \frac{\log ^+
\log^+ M(r,f)}{-\log(1- r)} \ge 1.  
$$
To be more precise, consider the canonical product
$$\mathcal{P}(z, (a_k), s):= \prod_{k=1}^\infty E(A(z, a_k), s),$$
where  $$E(w,s)=(1-w)\exp \{w+
w^2/2 +\dots+ w^s/s\},\; s\in \Z_+,$$ is the Weierstrass primary
factor, $A(z, \zeta)=\dfrac{1-|\zeta|^2}{1-z\bar\zeta}$, $z\in \mathbb{D}$, $\zeta\in \overline{\mathbb{D}}$.

Let $$\tilde\square(\rr\vfi):= {\mathcal R}\Bigl(\rr\vfi , \frac {1-r}2\Bigr),$$
$\nu(\rr\vfi)$ be the number of zeros of $\mathcal{P}$ in $\tilde\square(\rr\vfi)$.
We define
\begin{equation}\label{e:mu_ord}
    \nu_1(\vfi)=\limsup_{r\uparrow1} \frac{\log ^+ \nu(\rr\vfi, \mathcal{P})}{-\ln (1-r)}, \quad \nu[\mathcal{P}]=\sup_\vfi \nu_1(\vfi).
\end{equation}

 With the notation above we have (\cite[Theorem V]{L_rep})
\begin{gather}
    \label{e:rom}  \rho_M[\mathcal{P}] \begin{cases} =\nu[\mathcal{P}], & \rho_M[\mathcal{P}]\ge 1, \\
\le \nu[\mathcal{P}]\le 1, & \rho_M[\mathcal{P}]< 1. \end{cases}
\end{gather}


Given a Borel measure  $\mu$ on $\mathbb{D}$ satisfying
\begin{equation}\label{e:blashke_con_s}
\int_{\D} (1-|\zeta|)^{s+1}\, d\mu_f(\zeta)<\infty,  \quad s\in \mathbb{N}\cup \{ 0\},
\end{equation}
   define the canonical integral as
\begin{equation}\label{e:can_int}
U(z; \mu, s):=\int_{\mathbb{D}} \log | E(A(z, \zeta), s)| \, d\mu(\zeta).
\end{equation}

Let $(  q >-1)$ $$S_{\al}(z)=\Gamma(1+q)\Bigl(
 \frac 2{(1-z)^{q+1}} -1\Bigr), \quad P_q(z)
=\Re S_q(z), \quad S_q(0)=\Gamma(q+1).$$
Note that $S_0$ and $P_0$ are the Schwarz and Poisson kernels, respectively.

Let $u$ be  a subharmonic function in $\D$  of the  form
\begin{gather}
u(z)= U(z; \mu,s )-   \frac 1{2\pi}\int\limits_0^{2\pi}
{P_s(ze^{-i\theta})}{d\psi^*(\theta )}+ C, \label{e:u_bar_rep}
\end{gather}
where
$\psi^*\in BV[0,2\pi]$, $\mu$ is the Riesz measure of $u$ satisfying \eqref{e:blashke_con_s}.
Note that every subharmonic function $u$ of finite order in $\D$, i.e. satisfying $\log \max\{ u(z): |z|=r\}=O(\log \frac1 {1-r})$ $(r\uparrow 1)$, can be represented in the form \eqref{e:u_bar_rep} for an appropriate $s\in \N\cup\{0\}$ (\cite{HK}, \cite[Chap.9]{Dj}).

Let $M$ be  Borel's subset of
$\overline{\D}$ such that $M\cap \partial \mathbb{D}$ is measurable with respect to the Lebesgue measure on $\partial \D$. Let $u$ be  a subharmonic function in $\D$  of the  form \eqref{e:u_bar_rep}. We set
\begin{equation}\label{e:cms}
  \lambda_u(M)= \int_{\D\cap M} (1-|\zeta|)^{s+1}\, d\mu_u(\zeta) + \psi(M\cap \partial \D),
\end{equation}
where $\mu_u$ is the Riesz measure of $u$, $\psi$ is the (signed) Stieltjes  measure associated with $\psi^*$. Note that, in the case $u=\log|f|$ we have
 $\mu_{\log|f|}(\zeta)= \sum_{n} \delta(\zeta-a_n)$, where $(a_n)$ is the zero sequence of $f$.

Let $|\lambda|$ denote  the total variation of $\lambda$.
\begin{theorem}\label{t:u_finite}
Let $u$ be a subharmonic function in $\D$ of the form \eqref{e:u_bar_rep}, $\gamma \in (0,s+1]$, $p\in (1, \infty)$. Let $\lambda$ be defined by \eqref{e:cms}. If
\begin{equation}\label{e:lambda_p_mod}
    \biggl( \inp |\lambda|^p(\mathcal{C}(\vfi, \de))\, d\vfi \biggr)^{\frac 1p} =O(\delta^\gamma), \quad 0<\delta<1,
\end{equation}
holds, then
\begin{equation}\label{e:m_p_est}
    m_p(r,u)=O((1-r)^{\gamma-s-1}), \quad r\uparrow 1.
\end{equation}
\end{theorem}

\begin{theorem}\label{t:4}
Let $f$ be of the form
\begin{gather}
f(z)= C_q z^\nu \mathcal{P}(z, (a_k),q) \exp\Bigl\{ \frac 1{2\pi}\int\limits_0^{2\pi}
{S_q(ze^{-i\theta})}{d\psi^*(\theta )}\Bigl\}, \label{e:bar_rep}
\end{gather}
where    $\psi^*\in BV[0,2\pi]$, $(a_k)$ is the zero  sequence of
$f$ such that $\sum_k (1-|a_k|)^{q+1} <+\infty$, $\nu \in \Z_+$, $C_q\in \C$. Let $\gamma \in (0,s+1]$, $p\in (1, \infty)$. Let $\lambda$ be defined by \eqref{e:cms} for $u=\log |f|$. If
\begin{equation}\label{e:lambda_p_f}
    \biggl( \inp |\lambda|^p(\mathcal{C}(\vfi, \de))\, d\vfi \biggr)^{\frac 1p} =O(\delta^\gamma), \quad 0<\delta<1,
\end{equation}
holds,  then
\begin{equation}\label{e:m_p_est_f}
    m_p(r,\log|f|)=O((1-r)^{\gamma-s-1}), \quad r\uparrow 1.
\end{equation}
\end{theorem}

{\bf Remark.} Suppose that $\mu$ is 1-periodic measure on $\mathbb{R}$ finite on the compact Borel sets, and $p\ge 1$.
Then $$\biggl( \int_0^1 \Bigl(\mu(( x- \delta, x+\delta))\Bigr)^p \, dx \biggr)^\frac 1p=O(\delta ^{\frac 1p}), \quad \delta \in (0,1).$$
In fact,  assume that $\mu([0,1))=C$.  Then by Fubini's theorem
\begin{gather*}
\int_0^1 \Bigl(\mu(( x-\delta,x+ \delta))\Bigr)^p \, dx \le (2C)^{p-1}\int_0^1 \mu( (x-\delta,x+ \delta))\, dx =\\ =(2C)^{p-1}\int_0^1 \int _{(x-\delta, x+\delta)} d\mu(y) \, dx \le  \\
\le (2C)^{p-1} \int_{-\delta}^{1+\delta} d\mu (y) \int_{(y-\delta, y+\delta)} dx\le 2\delta \mu ((-\delta, 1+\delta))\le 3 (2C)^{p} \delta.
\end{gather*}
It follows from representation \eqref{e:u_bar_rep} that $\rho_1[u]\le s$. Then \eqref{e:m_p_est} implies $$\rho_p[u]\le s+1-\frac 1p.$$ It is known that this is a sharp inequality (\cite{Li_mp,Li_mp1}), in general. However, Theorems \ref{t:u_finite} and \ref{t:4} characterize classes where $\rho_p[u]$ takes a particular value. 

\smallskip

Examples in Section 4  show that the assertion of Theorem \ref{t:4} is sharp.

The following theorem provides a sharp estimate for means of canonical integrals or products in terms of growth of their Riesz measures.
\begin{theorem} \label{t:u_suf_groth_al}
Suppose that $u$ is of the form \eqref{e:can_int},
 $s \in \mathbb{N}\cup \{0 \}$, $p\in (1, \infty)$, and $\alpha>0$  are such that $\alpha+\frac 1p<s+1$.
If \eqref{e:nru_stolz} holds, then \eqref{e:u_m_p_est} is valid.
\end{theorem}


\section{Kernels $K_s(z,\zeta)$ and representation of functions of finite order}
We define $$K(z, \zeta)=\frac{G(z, \zeta)}{1-|\zeta|}=\frac1{1-|\zeta|} \log \Bigl|\frac{1-z\bar\zeta}{z-\zeta}\Bigr|, \quad z\in {\D}, \zeta\in \D, z\ne \zeta,
$$ where $G(z, \zeta)$ is the Green function for $\D$.
We have the following properties of $K(z, \zeta)$, $z=re^{i\varphi}$, $\zeta=\rho e^{i\theta}$.

\begin{proposition} \label{p:k_prop} The following hold:

a) $K(z, 0)=-\log |z|$.

b) $0\le K(z, \zeta)\le \frac{1-|z|^2}{|z-\zeta|^2}$.

c) If $  D \Subset  \D$, then  uniformly in $z\in D$
$$ \lim_{\rho\uparrow 1} K(z, \rho e^{i\theta}) =\frac {1-|z|^2}{|\rho e^{i\theta}-z|^2}=P_0(ze^{-i\theta}).$$

d) $$|K(z, \zeta)|\ge \frac 1{12} \frac{1-|z|^2}{|z-\zeta|^2}, \quad 1-|\zeta|\le \frac12 (1-|z|). $$
\end{proposition}
\begin{proof}[Proof of the proposition]
b) We have
\begin{gather*}
0\le K(r e^{i\vfi}, \rho e^{i\theta})=\frac 1{2(1-\rho)} \log  \frac {1-2r\rho \cos(\varphi-\theta)+ r^2\rho^2}{r^2-2r\rho \cos(\varphi-\theta)+ \rho^2}=\\ =\frac 1{2(1-\rho)} \log \Bl 1+\frac {(1-r^2)(1-\rho^2)}{r^2-2r\rho \cos(\varphi-\theta)+ \rho^2} \Br\le \\ \le \frac1{2(1-\rho)} \frac {(1-r^2)(1-\rho^2)}{r^2-2r\rho \cos(\varphi-\theta)+\rho^2}\le \frac{1-r^2}{|re^{i\varphi} -\rho e^{i\theta}|^2} .
   \end{gather*}

c) The assertion  easily follows from the equality
$$ K(z, \zeta)=\frac1{2(1-|\zeta|)}\log \Bigl( 1+ \frac{(1-|z|^2)(1-|\zeta|^2)}{|z-\zeta|^2}\Bigr)$$
see b).


d) It is proved  in \cite{Chy_Ska}.
\end{proof}

Due to d), we set $K(z, e^{i\theta}):=P_0(ze^{-i\theta})$ preserving continuity of $K$ on $\partial \D$ with respect to the second variable.

Let $s\in \N$.
We write
$$K_s(z, \zeta)=-\frac{\log |E(A(z, \zeta), s)|}{(1-|\zeta|)^{s+1}}, \quad \zeta\in \overline{\D}, z\in \D, z\ne \zeta,$$
i.e. $K(z,\zeta)=K_0(z, \zeta)$, we set   $K_s(z,z)=-\infty$, $z\in \D$.

Let  $D^*(z, \sigma)=\{\zeta :\bigl| \frac{z-\zeta}{1-z\bar \zeta}\bigr|<\sigma\}$ be the pseudohyperbolic disc with the center $z$ and radius $\sigma\in (0,1]$.
\begin{proposition} \label{p:k_prop} The following hold:
\begin{itemize}
 \item [i)]  \begin{equation}\label{e:k_s_far}
| K_s(z, \zeta)| \le  \frac{C(s)}{|1-z\bar\zeta|^{s+1}} , \quad \zeta  \not\in D^*(z,\frac17).
 \end{equation}
\item [ii)]  \begin{equation}\label{e:k_s_close}
| K_s(z, \zeta)| \le  \frac{C}{(1-|\zeta|)^{s+1}} \log\bigl| \frac{1-z\bar \zeta}{z-\zeta}\bigr|, \quad \zeta  \in D^*(z,\frac17).
 \end{equation}
\item [iii)] If $z\in D\Subset \D$, then $$K_s(z, \rho e^{i\theta})\rightrightarrows
\frac{2^{s+1}}{s+1}\Re \frac{1}{(1-z e^{-i\theta})^{s+1}}=\frac{2^sP_s(ze^{-i\theta})}{(s+1)!}+C(s), \quad \rho \uparrow 1.$$
\end{itemize}
 \end{proposition}
\begin{proof}[Proof of the proposition]
 The upper estimate for $K_s(z, \zeta)$
\begin{equation}\label{e:k_s_up}
K_s(z, \zeta)\le \frac{2^{s+2}}{(1-|\zeta|)^{s+1}} |A(z,\zeta)|^{s+1}\le \frac{2^{2s+3}}{|1-z\bar\zeta|^{s+1}}, z\in \D, \zeta \in \overline{\D},
\end{equation}
follows from the known estimate of the primary factor (\cite[Chap. V.10]{Tsuji}). Also,
\begin{equation}  \label{e:k_s_repr<1}
K_s(z,\zeta)=\frac{\Re \sum_{j=1}^\infty \frac1j(A(z, \zeta))^j}{(1-|\zeta|)^{s+1}},
\end{equation}
 provided that $|A(z, \zeta)|<\frac 12$, so  $$|K_s(z,\zeta)|\le  \frac{2|A(z, \zeta)|^{s+1}}{(s+1)(1-| \zeta)^{s+1}} \le \frac{2^{s+2}}{s+1} \frac{1}{|1-z\bar \zeta|^{s+1}}.$$ Hence,  it remains to consider the case when $|A(z,\zeta)|\ge \frac12$.

Since for all $z\in \D$, $\zeta \in \overline{\D}$ $|A(z, \zeta)|\le 2$,  we  have for $\zeta \not\in D^*(z, \frac17)$
 \begin{gather*}
    K_s(z, \zeta)(1-|\zeta|)^{s+1} =\log \Bigl| \frac{z-\zeta}{1-z\bar \zeta}\Bigr|+ \Re \sum_{j=1}^s \frac{A(z, \zeta)}{j}\ge -\log 7 -\sum_{j=1}^s \frac{2^j}j =\\ =-C(s)\ge -C(s)2^{s+1} |A(z, \zeta)|^{s+1}.
 \end{gather*}
 Hence,
 \begin{equation}\label{e:k_s_low_far}
 K_s(z, \zeta)\ge -\frac{C(s)}{|1-z\bar\zeta|^{s+1}}, \quad \zeta \not \in D^*(z,\frac17),
 \end{equation}
and i) follows.

Similar arguments give ii).

 Let us prove iii). If $z\in D\Subset \D$, then it follows from the representation \eqref{e:k_s_repr<1} that
 $$K_s(z, \rho e^{i\theta})\rightrightarrows
\frac{2^{s+1}}{s+1}\Re \frac{1}{(1-z e^{-i\theta})^{s+1}}=\frac{2^sP_s(ze^{-i\theta})}{(s+1)!}+C(s), \quad \rho \uparrow 1.$$
\end{proof}

Due to properties of $K_s(z, \zeta)$ the representation \eqref{e:u_bar_rep} could be rewritten in the form
(cf. \cite{Gr1}, \cite[Part II]{Gr68})
\begin{equation}
  u(z)=-\int_{\overline\D} K_s(z, \zeta) d\lambda(\zeta)+ C, \label{e:ks_rep}
\end{equation}
where $\frac{2^s}{(s+1)!} d\lambda(e^{i\theta}) =\frac1{2\pi} d\psi^*(\theta)$, and $\lambda=\lambda_u$ is defined  by \eqref{e:cms}.

Similarly,
for any $u\in \SH^\infty$ we have the following representation (cf. \eqref{e:uzob})
\begin{equation}
u(z)= -\int_{\bar{\mathbb{D}}} K(z,\zeta)\, d\lambda_u(\zeta)+ C.
\end{equation}

\smallskip 
{\bf Remark.} The idea of such representation  goes back  to  results of Martin \cite[Chap. XIV]{Br}.

\section{Proofs}


\begin{proof}[Sufficiency of Theorem \ref{t:u_bound}.]


We write
$$ u_1(z)=-\int_{D^*(z, \frac 17)}K(z, \zeta)\, d\lambda (\zeta),   \quad u_2(z)=-\int_{\bar{\mathbb{D}}\setminus D^*(z, \frac 17)}K(z, \zeta)\, d\lambda (\zeta).$$
Let us estimate $I_1=\int_{-\pi}^\pi |u_1(re^{i\varphi})|^p\, d\varphi$.

By the H\"older inequality
\begin{gather*}
 |u_1(re^{i\varphi})|=\int\limits_{D^*(re^{i\varphi}, \frac 17)}  \log \Bm \frac{1-re^{i\varphi}\bar \zeta}{re^{i\varphi}-\zeta}\Bm  d\mu(\zeta) \le  \\ \le \biggl( \int\limits_{D^*(re^{i\varphi}, \frac 17)} \Bigl( \log \Bm \frac{1-re^{i\varphi}\bar \zeta}{re^{i\varphi}-\zeta}\Bm \Bigr)^p \, d\mu(\zeta)\biggr)^\frac 1p \Bigl(\mu\Bigl(D^*\Bigl(re^{i\varphi}, \frac 17\Bigl)\Bigl)\Bigr)^{\frac{p-1}{p}},
\end{gather*}
hence
\begin{gather*}
 I_1\le \int_{-\pi}^\pi \biggl(  \int\limits_{D^*(re^{i\varphi}, \frac 17)} \Bm \log \Bm \frac{1-re^{i\varphi}\bar \zeta}{re^{i\varphi}-\zeta}\Bm \Bm^p \, d\mu(\zeta) \,
 \mu^{p-1}\Bl D^*\Bigl(re^{i\varphi}, \frac17\Bigl)\Br \biggr) \, d\varphi.
\end{gather*}
 Since $D^*(z, \frac{h}{2+h})\subset D(z, (1-|z|)h)$ (\cite{Illin}), with   $h=\frac 13$ we get $$D^*\Bigl(z, \frac 17\Bigl)\subset D\Bigl(z, (1-|z|)\frac 13\Bigl)\subset \Box\Bigl(z, \frac13(1-|z|)\Bigl),$$
 where $\Box(re^{i\varphi}, \sigma)=\{\rho e^{i\theta}: |\rho-r|\le \sigma, |\theta-\varphi|\le \sigma\}$.
Therefore, using Fubini's theorem, we deduce
\begin{gather*}
 I_1\le \int_{-\pi}^\pi \biggl( \int_{\Box(z, \frac 13 (1-r))} \Bigl( \log \Bm \frac{1-re^{i\varphi}\bar \zeta}{re^{i\varphi}-\zeta}\Bm \Bigr)^p\biggr) \mu^{p-1}\Bigl(\Box\bigl(z, \frac 13 (1-r)\bigr)\Bigr)  \, d\mu(\zeta) d\varphi\le \\ \le
\int_{-\pi}^\pi \biggl( \int_{\Box(z, \frac 13 (1-r))} \Bigl( \log \Bm \frac{1-re^{i\varphi}\bar \zeta}{re^{i\varphi}-\zeta}\Bm \Bigr)^p\biggr) \mu^{p-1}\Bl\Box\bigl(re^{i\arg\zeta}, \frac 23 (1-r)\bigr)\Br  \, d\mu(\zeta) d\varphi =\\
\le \iint\limits_{\begin{substack} {-\pi-\frac{1-r}3 \le \theta\le \pi+\frac{1-r}3\\ |\rho-r|\le \frac{1-r}3\\|\theta-\vfi|\le \frac{1-r}3} \end{substack}} \Bigl(\log \Bm \frac{1-r\rho e^{i(\varphi-\theta)}}{re^{i\varphi}-\rho e^{i\theta}}\Bm \Bigr)^p \mu^{p-1}\Bl\Box\bigl(re^{i\theta}, \frac 23 (1-r)\bigr)\Br  \, d\mu(\rho e^{i\theta}) d\varphi = \\
\le
2\int\limits_{||\zeta|-r|\le \frac 13 (1-r)}\mu^{p-1}\Bl\Box\bigl(re^{i\arg\zeta}, \frac 23 (1-r)\bigr)\Br \int_{-\pi}^\pi  \Bigl( \log \Bm \frac{1-re^{i\varphi}\bar \zeta}{re^{i\varphi}-\zeta}\Bm \Bigr)^p \, d\varphi \, d\mu(\zeta).
\end{gather*}
We know that (\cite{5}) for any $a,b \in \mathbb{C}$, and $p>1$  $$\int_{-\pi}^\pi \Bm \log \Bm \frac{a-e^{i\theta}}{b-e^{i\theta}} \Bm\Bm^p \, d\theta\le C(p)|a-b|$$  holds. Using this inequality we obtain $(r\in (\frac 12, 1))$
\begin{gather}
 I_1 
\le 4 C(p)(1-r)\int\limits_{||\zeta|-r|\le \frac 13(1-r)} \mu^{p-1}\Bl \Box \bigl(re^{i\arg\zeta}, \frac 23 (1-r)\bigr)\Br d\mu(\zeta).\label{e:key_place}
\end{gather}
In order to proceed we need the following lemma.
\begin{lemma} \label{l:mes_p_est}
 Let $\nu $ be  a $2\pi$ periodic positive Borel measure on $\mathbb{R}$, $p\ge 1$, $\delta \in (0,\pi)$. Then
\begin{equation}
 \label{e:measure_p_est}
\int_{[-\pi,  \pi)} \nu^{p-1}((\theta-\delta, \theta+\delta)) d\nu(\theta) \le \frac{2^{p+1}}{\delta} \int_{[-\pi, \pi )} \nu^p((\theta-\delta, \theta+\delta)) d\theta.
\end{equation}
\end{lemma}

\begin{proof}[Proof of Lemma \ref{l:mes_p_est}]  First, we prove \eqref{e:measure_p_est} for $p=1$\footnote{The author thanks Prof. Sergii Favorov for the idea of the proof of this lemma.}.

We have
\begin{gather}
 \nonumber
\int_{[-\pi, \pi)} d\nu(\theta)=\int_{[-\pi, \pi)} \frac 1\delta \int_{\theta-\frac \delta 2}^{\theta+\frac \delta 2} dx d\nu(\theta) \le
\int_{[-\pi-\frac{\delta}{2}, \pi+\frac{\delta}{2})} dx \int_{[x-\frac{\delta}2, {x+\frac\delta 2})} \frac 1\delta d\nu (\theta)= \\ =\int_{[-\pi-\frac{\delta}{2}, \pi+\frac{\delta}{2})}   \frac {\nu ([x-\frac\delta 2, x+\frac \delta 2))}\delta dx \le 2 \int_{[-\pi, \pi)}  \frac {\nu ([
x-\frac\delta 2, x+\frac \delta 2))}\delta dx\le \label{e:p=1delta2} \\
\le  2 \int_{[-\pi, \pi)}  \frac {\nu ((x-\delta , x+ \delta ))}\delta dx
. \label{e:p=1}\end{gather}

We now consider arbitrary finite $p>1$. Applying \eqref{e:p=1delta2} with $d\nu_1(\theta)=\nu^{p-1} ((\theta-\delta, \theta+\delta)) d\nu(\theta)$,  we get
\begin{gather*}
 \int\limits_{[-\pi, \pi)} \nu^{p-1}((\theta-\delta, \theta+\delta)) d\nu(\theta) =\int\limits_{[-\pi, \pi)} d\nu_1(\theta)
\le 2\int\limits_{[-\pi, \pi)} \frac{\nu_1([x-\frac \delta 2, x+\frac{\delta}{2}))}{\delta} dx =\\
= 2\int\limits_{[-\pi, \pi)}  \int_{[x-\frac\delta2, x+\frac \delta 2)} \nu^{p-1}((\theta-\delta, \theta+\delta)) d\nu(\theta) dx\le \\ \le  2\int_{[-\pi,\pi)} \frac{\nu^{p-1}((x-\frac {3\delta} 2, x+\frac{3\delta}{2})) \nu([x-\frac \delta 2, x+\frac{\delta}{2} ))}{\delta} dx\le \\
\le  2\int_{[-\pi,\pi)} \frac{\nu^p((x-\frac {3\delta} 2, x+\frac{3\delta}{2})) }{\delta} dx\le \\
\le   2\int_{[-\pi,\pi)} \frac{\nu^p((x-\frac {3\delta} 2, x)\cup [x, x+\frac{3\delta}{2})  ) }{\delta} dx\le  \\ \le 2^{p} \int_{[-\pi, \pi) } \frac{\nu^p((x-\frac{3\delta}2, x))}{\delta} dx+ 2^{p} \int_{[-\pi, \pi) } \frac{\nu^p([x, x+\frac{3\delta}2, x))}{\delta} dx\le \\
\le 2^{p+1} \int_{[-\pi, \pi)} \frac{\nu^p((x-{\delta}, x-{\delta}))}{\delta} dx.
\end{gather*} The lemma
 is proved.
\end{proof}
Let us continue the proof of the sufficiency.

We denote the nondecreasing  function  $$N_r(\theta)=\lambda (\{ \rho e^{i\alpha} : |r-\rho|\le \frac 23 (1-r), -\pi \le \alpha\le \theta\}), \quad \theta\in [-\pi, \pi). $$ We extend it on the real axis preserving monotonicity  by  $N_r(x+2\pi)-N_r(x)=N_r(2\pi)-N_r(0)$, $x\in \mathbb{R}$. Let $\nu_r$ be the corresponding Stieltjes measure on $\mathbb{R}$.
Estimate \eqref{e:key_place} can be written in the form
\begin{gather*}
I_1 \le \frac{C}{(1-r)^{p-1}}\int\limits_{||\zeta|-r|\le \frac 13(1-r)} \lambda^{p-1}\Bl\Box\bigl( re^{i\arg\zeta}, \frac 23 (1-r)\bigr) \Br d\lambda(\zeta)= \\
=\frac{C}{(1-r)^{p-1}} \int_{-\pi}^\pi \nu_r^{p-1}\Bigl(\Bigl[\theta-\frac23 (1-r), \theta+\frac 23(1-r)\Bigr]\Bigr) d\nu_r(\theta)\le \\
\le 2^{p+1} \frac{3C}{2(1-r)^p} \int_{-\pi}^\pi \nu_r^{p}\Bigl(\Bigl[\theta-\frac23 (1-r), \theta+\frac 23(1-r)\Bigl]\Bigl) d\theta\le \\
 \le \frac{C(p)}{(1-r)^{p}}
\int_{-\pi}^\pi \lambda^{p}\Bl \Box\bigl(re^{i\theta}, \frac 23 (1-r)\bigr) \Br \, d\theta \le \frac{C}{(1-r)^{p}} (1-r)^{p\gamma}.
\end{gather*}
We have used Lemma \ref{l:mes_p_est} and the assumption of the theorem on the complete measure.

Thus, we have  \begin{equation} \label{e:phi_1_est}
                    \Bigl(\int_{-\pi}^\pi |u_1(re^{i\varphi})|^p\, d\varphi \Bigr)^{\frac 1p} \le C(p) {(1-r)^{\gamma-1}}.
                   \end{equation}

Let us estimate $u_2(z)=-\int_{\mathbb{\bar D}} K(z, \zeta) d\tilde \lambda(\zeta)$, where
$d\tilde \lambda(\zeta)=\chi_{\bar {\mathbb{D}}\setminus D^*(z, \frac 17)}(\zeta) d\lambda(\zeta)$.

Since  $\mathop{\rm supp} \tilde\lambda \cap  D^*(z, \frac17)=\varnothing$, for $\zeta\not\in D^*(z, \frac17)$ we have by Proposition 1 that
\begin{equation}\label{e:k_kern_2est}
K(z, \zeta)\le \frac{49(1-|z|^2)}{|1 -z\bar \zeta|^2}.
\end{equation}
Let $E_n=E_n(re^{i\varphi})=\mathcal{C}(\varphi, 2^n(1-r))$, $n \in \N$,
 $E_0=\varnothing$.
Then for $\zeta =\rho e^{it} \in \D\setminus E_n(z)$, $n\ge 1$, we have
\begin{gather*}
    |1-\rho r e^{i(\varphi-t)}|\ge | 1- \rho  e^{i(\varphi-t)}| -\rho(1-r)\ge 2^n(1-r) -(1-r)\ge 2^{n-1}(1-r).
\end{gather*}
and $|1-\rho r e^{i(\varphi-t)}|\ge 1-r\rho \ge 1-r$ for $\zeta \in E_1(z)$.
Therefore  $(\frac 1p+\frac 1{p'}=1)$ 
\begin{gather} \nonumber
|u_2(re^{i\varphi})|^p \le
\biggl(\biggl(  \sum_{n=1}^{[\log _2 \frac 1{1-r}]} \int_{E_{n+1}\setminus E_{n}} + \int_{E_1} \biggr)  \frac{49(1-r^2)}{|1-r e^{i
 \varphi}\bar \zeta|^2} d\tilde\lambda(\zeta) \biggr)^p\le \\ \nonumber
 \le   49 ^p \biggl( \sum_{n=1}^{[\log _2 \frac 1{1-r}]} \int_{E_{n+1}\setminus E_{n}} \frac{2(1-r)}{(2^{n-1}(1-r))^2} d\tilde \lambda(\zeta)+
 \int_{E_1}   \frac{2}{1-r} d\tilde\lambda(\zeta) \biggr)^p < \\
 \le  \Bigl(\frac{400}{1-r}\Bigr)^p \sum_{n=1}^{[\log _2 \frac 1{1-r}]+1} \frac{\tilde \lambda^p (E_{n}(z))}{2^{\frac 32np}} \biggl( \sum_{n=1}^\infty \frac1{2^{\frac {np'}2}}\biggr)^{\frac p{p'}} \le \frac{C(p)}{(1-r)^p} \sum_{n=1}^\infty  \frac{\tilde \lambda^p (E_{n}(z))}{2^{\frac 32np}}.
 \label{e:2star_est}
\end{gather}

It follows from the latter inequalities and the assumption of the theorem that ($r\in [\frac 12,1)$)
\begin{gather*}  \inp |u_2(\rr\vfi)|^p\, d\vfi \le \frac {C( p)}{(1-r)^p} \sum_{n=1}^\infty \int_{0}^{2\pi} \frac{\tilde\lambda^p(E_n(re^{i\varphi}))}{2^{\frac 32 np}} \, d\varphi \le \\
\le \frac{C(p)}{(1-r)^p} \sum_{n=1}^\infty \frac{(2^n(1-r))^{p\gamma}}{2^{\frac 32 np}}=\frac {C(p)}{(1-r)^{p(1-\gamma)}} \sum_{n=1}^\infty 2^{np(\gamma -\frac 32)}= \frac{C(p,\gamma)} {(1-r)^{p(1-\gamma)}}.\end{gather*}

$$ \Bl \inp |u_2(\rr\vfi)|^p\, d\vfi \Br^{\frac 1p}\le \frac {C(\gamma, p)}{(1-r)^{1-\gamma}}, \quad r\in [0,1).$$
Sufficiency of Theorem 1 is proved.

{\it Necessity. }
Using property d) of Proposition 1, we obtain
$$ |u(\rr\theta)|\ge \int_{\mathcal{C}(\vfi, \frac{1-r}{2})} K(\rr\vfi, \zeta) \, d\lambda(\zeta)\ge
\frac 1{12} \int_{\mathcal{C}(\vfi, \frac{1-r}{2})} \frac{1-r^2}{|\rr\vfi -\zeta|^2}\, d\lambda(\zeta).$$
Elementary geometric arguments  show that $|\rr\vfi -\rho e^{i\theta}| \le |\rr\vfi -e^{i\theta}|$ for $1>\rho\ge r\ge 0$.
It then follows that
\begin{gather*}
 |u(\rr\theta)|\ge\frac 1{12} \int_{\mathcal{C}(\vfi, \frac{1-r}{2})} \frac{1-r^2}{|\rr\vfi -e^{i\theta}|^2}  d\lambda(\rho e^{i\th}) \ge \\ \ge \frac 1{3(\frac{\pi^2}{4}+1)} \frac{1-r^2}{(1-r)^2} \int_{\mathcal{C}(\vfi, \frac{1-r}{2})} d\lambda(\rho e^{i\theta})\ge \frac{\lambda(\mathcal{C}(\vfi, \frac{1-r}{2}))}{3(\frac{\pi^2}{4}+1)(1-r)}.
\end{gather*}
By the assumption of the theorem we deduce that
$$ \frac{C}{(1-r)^{(1-\gamma)p}} \ge \inp |u(\rr\vfi)| ^p \, d\vfi\ge C \frac{\inp \lambda^p (\mathcal{C}(\vfi, \frac{1-r}{2}))\, d\vfi}{(1-r)^p}.$$
Hence $\inp \lambda^p (\mathcal{C}(\vfi, \frac{1-r}{2}))\, d\vfi =O((1-r)^{\gamma p})$  as $r\uparrow 1$.
This completes the proof of necessity.
\end{proof}

\begin{proof}[Proof of Theorem 3]
Due to \eqref{e:u_lambda_p} we write 
\begin{gather*}
u(z)=-\int_{\bar{\mathbb{D}}} K_s(z,\zeta)\, d\lambda(\zeta)=
\\ =-\int_{D^*(z, \frac 17)}K_s(z, \zeta)\, d\lambda (\zeta)-\int_{\bar{\mathbb{D}}\setminus D^*(z, \frac 17)}K_s(z, \zeta)\, d\lambda (\zeta)\equiv u_1+u_2.
\end{gather*}
According to \eqref{e:k_s_close}  $$|u_1(z)| \le C(s) \int_{D^*(z, \frac 17)}  \log\Bigl| \frac{1-z\bar \zeta}{z-\zeta}\Bigr| d\mu (\zeta).$$ Its estimate repeats that for the case $s=0$.

Let us estimate $p$th means of $u_2(z)$. Using Proposition 2 we deduce (cf. proof of Theorem 1)
\begin{gather} \nonumber
|u_2(re^{i\varphi})|^p \le
\biggl(\biggl(  \sum_{n=1}^{[\log _2 \frac 1{1-r}]} \int_{E_{n+1}\setminus E_{n}} + \int_{E_1} \biggr)  \frac{C(s)}{|1-r e^{i
 \varphi}\bar \zeta|^{s+1}} |d\tilde\lambda(\zeta)| \biggr)^p\le \\ \nonumber
 \le  \frac{C}{(1-r)^{(s+1)p}}\sum_{n=1}^{[\log _2 \frac 1{1-r}]+1} \frac{|\tilde \lambda|^p (E_{n}(z))}{2^{(s+\frac12)np}} \biggl( \sum_{n=1}^\infty \frac1{2^{\frac {n{p'}}2}}\biggr)^{\frac p{p'}} \le  \\ \le \frac{C}{(1-r)^{(s+1)p}} \sum_{n=1}^\infty  \frac{|\tilde \lambda|^p (E_{n}(z))}{2^{(s+\frac 12)np}}.
 \label{e:2star_est}\nonumber
\end{gather}
It follows from the latter inequalities and the assumption of the theorem that
\begin{gather*}  \inp |u_2(\rr\vfi)|^p\, d\vfi \le \frac {C( p)}{(1-r)^{(s+1)p}} \sum_{n=1}^\infty \int_{0}^{2\pi} \frac{|\tilde\lambda|(E_n(re^{i\varphi}))}{2^{(s+\frac12)np}} \, d\varphi \le \\ \le
\frac{C(p)}{(1-r)^{(s+1)p}} \sum_{n=1}^\infty \frac{(2^n(1-r))^{p\gamma}}{2^{(s+\frac 12) np}}
 = \frac{C(p,\gamma)} {(1-r)^{p(s+1-\gamma)}}, \quad r\in \Bigl[\frac 12,1\Bigr).\end{gather*}
Finally,
$$ \Bl \inp |u_2(\rr\vfi)|^p\, d\vfi \Br^{\frac 1p}\le \frac {C(\gamma, p)}{(1-r)^{s+1-\gamma}}, \quad r\in \Bigl [\frac 12,1\Bigr).$$
\end{proof}

\begin{proof}[Proof of Theorem \ref{t:u_supp_stoz}]
   Without loss of generality we assume that $\supp \mu_u\subset \{ z\in \overline{\D}: |1-z|<2 (1-|z|) \}=: \triangle$.

   {\it Neccesity.}
  Note that $\mathcal{R}(1-\delta, 2\delta)\subset \mathcal{C}(\vfi, \delta)$ for $\vfi\in [-\delta, \delta]$. Applying
   Theorem \ref{t:u_bound} we obtain
   $$\biggl( \int_{-\delta}^\delta  \lambda^p(\mathcal{R}(1-\delta, 2\de))\, d\vfi \biggr)^{\frac 1p} =O(\delta^{1-\alpha}), \quad 0<\delta<1, $$
   or $$ \mu_u(\mathcal{R}(1-\delta,2 \de))=O(\delta^{-\alpha -\frac 1p}), \quad 0<\delta<1.$$
   Since
   \begin{equation}\label{e:triang_incl}
     \triangle \subset \overline{D(0, \frac12)}\cup \bigcup_{n=1}^\infty \mathcal{R}(1-2^{-n}, 2^{1-n})
   \end{equation}
   we deduce
   $$ n(1-2^{-k}, u)\le C\sum_{n=1}^k 2^{n(\alpha+\frac 1p)}+ C=O(2^{k(\alpha+\frac 1p)}), \quad k\in \mathbb{N},$$
and the assertion follows.

 {\it Sufficiency.}
 It follows from the assumptions that
 $$\lambda (\mathcal{R}((1-\delta)e^{i\vfi}, 4\delta))=O(\delta^{1-\alpha-\frac 1p}), \quad \delta \downarrow 0.$$
 Then
 $$ \lambda(\mathcal{C}(\vfi, \delta))\le  \lambda \Bigl(\bigcup_{n=0}^\infty \mathcal{R}(1-\frac \delta{2^{n}}e^{i\vfi}, \frac{4\delta }{2^{n}})\Bigr) \le C\sum_{n=0}^\infty \Bigl(\frac{\delta}{2^n}\Bigr)^{1-\alpha-\frac 1p}= O(\delta^{1-\alpha-\frac 1p}),  \delta \downarrow 0.$$
 Since   $\supp \mu_u\subset  \triangle$, we have
 \begin{gather*}
   \int_{-\pi}^{\pi} \lambda ^p(\mathcal{C}(\vfi, \delta))\, d\vfi = \int_{-2\pi\delta}^{2\pi\delta} \lambda ^p(\mathcal{C}(\vfi, \delta))\, d\vfi =O(\delta \delta^{p(1-\alpha -\frac 1p)})=O(\delta^{p(1-\alpha)}), 
   \delta \downarrow 0. \end{gather*}
   It remains to apply Theorem \ref{t:u_bound}.
 \end{proof}

 \begin{proof}[Proof of Theorem \ref{t:u_suf_groth_al}]
  We confine ourselves to the case $s=0$. We keep the notation from the proof of Theorem \ref{t:u_bound}.
   It follows from  estimate \eqref{e:key_place} that

   \begin{equation}\label{e:u_1_est}
   \int_{-\pi}^\pi |u_1(re^{i\varphi})|^p\, d\varphi \le (1-r)n^{p-1}\Bl(r+\frac{2}{3}(1-r),u\Br) n\Bl(r+\frac{1}{2}(1-r),u\Br)=O\Bigl((1-r)^{-\alpha p}\Bigr).
\end{equation}
Let us estimate $p$th mean of $u_2$. We use estimate \eqref{e:k_kern_2est}, integral Minkowski's  inequality (\cite[\S A1]{stein}), standard estimates,  and integration by parts
\begin{gather*}
  \Bigl(\int_{-\pi}^\pi |u_2(re^{i\varphi})|^p\, d\varphi\Bigr)^{\frac 1p} \le  C \Bigl(\int_{-\pi}^\pi\Bigl( \int_{\D} \frac{1-r^2}{|1-re^{i\vfi}\bar \zeta|^2} \, d \lambda (\zeta)\Bigr)^p\, d\varphi\Bigr)^{\frac 1p}\le \\
  \le C \int_{\D}  \Bigl(\int_{-\pi}^\pi \Bigl(\frac{1-r^2}{|1-re^{i\vfi}\bar \zeta|^2} \Bigr)^p \, d\vfi \Bigr)^{\frac 1p}\, d\lambda(\zeta)\le C \int_{\D}   \frac{1-r}{(1-r| \zeta|)^{2-\frac 1p}} \, d\lambda(\zeta)= \\
  =C(1-r) \int_0^1 \frac{(1-t)dn(t,u)}{(1-rt)^{2-\frac 1p}}\le C(1-r) \biggl(\int_0^r \frac{dn(t,u)}{(1-t)^{1-\frac 1p}} + \\ +\int_r^1 \frac{(1-t)dn(t,u)}{(1-r)^{2-\frac 1p}} \biggr)\le  \frac{C}{(1-r)^{1-\frac1p}} \int_r^1 n(t,u)dt = O((1-r)^{-\alpha}), \quad r\uparrow 1.
  \end{gather*}
  Taking into account \eqref{e:u_1_est}, we obtain desired estimate.
 \end{proof}

\section{Examples}
{\bf Example 1.}
Following Linden \cite[Lemma 1]{Li_mp1}, given $\alpha \ge 1$,  $\beta \in [0,1]$, we consider the sequence of complex numbers
\begin{equation}\label{e:akm}
a_{k,m}=(1-2^{-   k})e^{im2^{-k}}, \quad 1\le m\le [2^{k\beta}]
\end{equation}
where each of numbers (\ref{e:akm}) is counted $[2^{\alpha k}]$ times.  Then  for $\mathcal{P}(z)=\mathcal{P}(z, (a_{k,m}), s)$, where $s=\min\{ q\in \mathbb{N}: q> \alpha+\beta-1\}$
 we have (see \cite{Li_mp1})
$ n(r, \mathcal{P})\asymp \Bigl(\frac 1{1-r} \Bigr)^{\alpha+\beta}$, $ \nu(r, \mathcal{P})\asymp \Bigl(\frac 1{1-r} \Bigr)^{\alpha}, \quad r\uparrow1.$
Therefore, by Theorem A \cite{L_rep} $\rho_M[\mathcal{P}]=\alpha$. In \cite{Li_mp1}  it is proved that
$\rho_p[\log|\mathcal{P}|]=\alpha+\frac{\beta-1}p$. We are going to prove that
\begin{equation}\label{e:lambda_p_lower-est}
    \biggl( \inp \lambda^p(\mathcal{C}(\vfi, \de))\, d\vfi \biggr)^{\frac 1p}\ge C(\delta^{s+1-\alpha-\frac{\beta-1}p}), \quad \delta \downarrow 0.
\end{equation}
It would imply that restriction  \eqref{e:lambda_p_mod} could not be weakened.

We first assume that $\beta\in (0,1)$.
Given $\delta \in (0,\delta_0)$ we define $\varphi_\delta=\delta^{1-\beta}-\pi\delta$, where $\delta_0$ is chosen such that $\varphi_\delta>0$.
Note that $\vfi_\delta\sim \delta^{1-\beta}$, $\delta \downarrow 0$.
According to the definition of $\mathcal{C}(\varphi, \delta)$,  $a_{k,m}\in \mathcal{C}(\varphi, \delta)$ if and only if
\begin{equation}\label{e:belong_cfi_deta}
    1-|a_{k,m}|=2^{-k}\le \delta, \quad \varphi -\pi \delta\le  m 2^{-k} \le  \varphi+\pi \delta.
\end{equation}
Let $G(\varphi, \delta)$ denote the set of ${(k,m)}$ such that \eqref{e:belong_cfi_deta} is valid.
It is easy to check that for $\varphi\in (0, \varphi_\delta)$ the set $G(\varphi, \delta)$ is not empty.
Let $$k_1(\varphi)=\min\{ k: 2^{-k} [2^{\beta k}]\le \varphi+ \pi \delta\},$$ where  $\varphi\in (0, \varphi_\delta)$.
Since $k_1(\varphi)$ tends to infinity uniformly with respect to $\varphi \in (0, \varphi_\delta)$ as $\delta \downarrow 0$, one can choose $\delta_1$ so small that for all   $\delta \in (0, \delta_1)$,  $ \varphi\in (0, \varphi_\delta)$ and $k\ge k_1(\varphi)$ the inequality $\frac{2^{-\beta k}}{ (1-\beta)(1-2^{\beta k})\log 2}\le 1$ holds.
Under this assumptions we deduce subsequently from the definition of $k_1=k_1(\varphi)$ that
$$|2^{k_1}(\varphi+\pi \delta)- 2^{\beta k_1}|< 1,$$
\begin{equation}\label{e:k_1_asym}
\frac {1-2^{-\beta k_1}}{\varphi+\pi \delta}< 2^{k_1(1-\beta)} < \frac {1+2^{-\beta k_1}}{\varphi+\pi \delta},
\end{equation}
$$ |k_1- \frac{1}{1-\beta} \log_2 \frac 1{\varphi+\pi \delta}|< 1.$$
It follows from the definition of $\vfi_\delta$ and  \eqref{e:k_1_asym} that
\begin{equation}\label{e:k_1low_est}
2^{k_1} >\frac{(1-2^{-k_1\beta})^{\frac 1{1-\beta}}}\delta>\frac2{\pi\delta} , \quad 0<\vfi <\vfi_\delta,
\end{equation}
Then, according to \eqref{e:belong_cfi_deta}, \eqref{e:k_1low_est} for $\delta\in (0, \min\{\delta_0, \delta_1\})$ and $\vfi\in (\frac12 \vfi_\delta, \vfi_\delta)$,   $\delta\downarrow 0$
\begin{gather}\nonumber
    \lambda(\mathcal{C}(\varphi, \delta))=\sum_{(k,m)\in G(\varphi, \delta)} [2^{\alpha k}] 2^{-k(s+1)} \ge     \sum_{m=[2^{k_1} (\varphi-\pi \delta )]+1}^ {[2^{k_1} (\varphi+\pi \delta )]} [2^{\alpha k_1}] 2^{-k_1(s+1)}\ge \\ \ge [2^{\alpha k_1}] 2^{-k_1(s+1)} (2^{k_1}  2\pi \delta-2)\ge [2^{\alpha k_1}] 2^{-k_1s}  \pi \delta\ge \frac {\pi\delta} 2 2^{(\alpha -s) k_1} \sim \frac{\pi \delta}{2} \Bigl(\frac 1\varphi\Bigr)^{\frac{\alpha -s}{1-\beta}}.  \label{e:com_meas_low_est}
    \end{gather}
It follows from the last estimate that
\begin{gather*}
\biggl(    \int_0^{2\pi}(\lambda (\mathcal{C}(\varphi, \delta)))^p d\varphi\biggr)^{\frac 1p} \ge \frac {\pi\delta} 2 \biggl(\int_{\varphi _\delta/2}^{\varphi _\delta} ( \varphi ^{\frac{s-\alpha}{1-\beta}})^p d\varphi\biggr)^{\frac1p} = \\
=\frac{\pi}{2(\frac{s-\alpha}{1-\beta}p+1)} \delta \varphi ^{\frac{s-\alpha}{1-\beta} +\frac 1p}\Bigr|^{\varphi _\delta}_{\varphi _\delta/2}\sim C(s, \alpha, p)\delta ^{1+s -\alpha-\frac{\beta-1}p}, \quad \delta\downarrow 0.
\end{gather*}

In the case $\beta=1$ the arguments could be simplified. By the  choice of $s$, $s>\alpha$. For $0< \vfi\le \frac 12$ , according to \eqref{e:belong_cfi_deta} we deduce
\begin{gather*}
  \lambda(\mathcal{C}(\varphi, \delta))=\sum_{k=[\log_2 \frac 1\delta]+1}^\infty    \sum_{m=[2^{k} (\varphi-\pi \delta )]+1}^{[2^{k} (\varphi+\pi \delta )]} [2^{\alpha k}] 2^{-k(s+1)}
   \ge\\ \ge \sum_{k=[\log_2 \frac 1\delta]+1}^\infty  [2^{\alpha k}] 2^{-k(s+1)}(2^{k}  2\pi \delta-2) \ge
   \sum_{k=[\log_2 \frac 1\delta]+1}^\infty  [2^{\alpha k}] 2^{-ks}\pi \delta\ge \\ \ge   \frac {\pi\delta} 2 \sum_{k=[\log_2 \frac 1\delta]+1}^\infty  2^{(\alpha -s) k} \asymp  \delta^{1+s-\alpha}.
   \end{gather*}
      Hence,
   \begin{gather*}
\biggl(    \int_0^{2\pi}(\lambda (C(\varphi, \delta)))^p d\varphi\biggr)^{\frac 1p} \ge  C(s, \alpha, p)\delta ^{1+s -\alpha}, \quad \delta\downarrow 0.
\end{gather*}
If $\beta=0$, then all zeros $a_k=1-2^{-k}$ are located  on $[0,1)$, and $s>\alpha-1$. For $\vfi\in (-\pi\delta, \pi\delta)$ we then have
according to \eqref{e:belong_cfi_deta}
\begin{gather*}
  \lambda(C(\varphi, \delta))=\sum_{k=[\log_2 \frac 1\delta]+1}^\infty    \ [2^{\alpha k}] 2^{-k(s+1)}
 \ge  C\sum_{k=[\log_2 \frac 1\delta]+1}^\infty  2^{(\alpha -s-1) k} \asymp  \delta^{1+s-\alpha}.
   \end{gather*}
Then
\begin{gather*} \biggl(    \int_0^{2\pi}(\lambda (C(\varphi, \delta)))^p d\varphi\biggr)^{\frac 1p} \ge C \delta ^{1+s -\alpha} \Bigl(\int_{-\pi \delta}^{\pi \delta}  d\vfi \Bigr)^\frac{1}{p} = C(s, \alpha, p) \delta ^{1+s -\alpha+\frac 1p}, \quad \delta\downarrow 0
\end{gather*}
as required.

\medskip
{\bf Example 2.}
Let $f(z)=\exp\Bigl\{\Bigl(\frac{1}{1-z}\Bigr)^{q+1}\Bigr\}$, $q>-1$, $f(0)=e$.
In this case $f(z)$ is of the form \eqref{e:bar_rep} with $(a_k)=\varnothing $, $\psi^*(\theta)=H(\theta)m_0$, $m_0>0$, where $H(\theta)$ the Heaviside function, i.e. $\lambda(\zeta) =m_0 \delta(\zeta-1)$. It is easy to check that
\begin{gather*} \biggl(    \int_0^{2\pi}(\lambda (C(\varphi, \delta)))^p d\varphi\biggr)^{\frac 1p}  =m_0 (2\pi \delta) ^\frac{1}{p},
\end{gather*}
and $m_p(r, \log|f|)\asymp (1-r)^{\frac 1p-q-1}$.

%


\begin{thebibliography}{19}
\bibitem{Br} M. Brelo.  On topologies and boundaries in potential theory, Springer-Verlag, Berin-Heidelberg-New York, 1971, Lect.Notes in Mathemetics, V.175.

    \bibitem{BrOC} J. Brune, J. Otrega-Cerd\'a, On $L^p$-solutions of the Laplace equation and zeros of holomorphic functions,
    Annali della Scuola Normale Superiore di Pisa -- Classe di Scienze {\bf 24}  (1997), no. 3,  571--591.

\bibitem{Chmm} I.Chyzhykov, A generalization of Hardy-Littlewood's
theorem, Math. Methods and Physicomechanical Fields, {\bf 49}
(2006), no.2, 74--79  (in Ukrainain)

\bibitem{Chy08ms} I.\  E.\ Chyzhykov.\ {Growth of  analytic  functions in the unit disc
and complete measure in the sense of Grishin}, Mat. Stud. {\bf 29} (2008), no. 1, 35--44.

\bibitem{Illin} I. Chyzhykov, Argument of bounded analytic functions and Frostman's type conditions, Ill. J. Math. {\bf 2} (2009), no.2, 515--531.

\bibitem{5}
I.\ Chyzhykov, Zero distribution and factorisation of analytic functions of slow growth in the unit disc, Proc. Amer. Math. Soc., {\bf 141}  (2013), 1297--1311.

\bibitem{Chy_Ska}
I.\ Chyzhykov, S.\ Skaskiv, Growth, zero distribution and factorization of analytic functions of moderate growth in the unit disc, Blaschke products and their applications, Fields Inst. Comm., {\bf 65}  (2013),  159--173.


\bibitem{CL} E.F.Collingwood, A.J.Lohwater, The theory of cluster
sets, Cambridge Univer. Press, 1966.
%
%


 \bibitem{Dj} Djrbashian M.M. Integral transforms and
 representations of functions in the complex domain, Moscow, Nauka, 1966 (in Russian).

\bibitem{Dj73} {M. M. Djrbashian}, {\it Theory of factorization and boundary properties of functions meromorphic in the disc,\/}  Proceedings of the ICM, Vancouver, BC, 1974.

\bibitem{Dur}  P. L. Duren,  Theory of $H^p$ spaces, Academic press, NY and London, 1970. -- 258 pp.

\bibitem{Gar88} S. J. Gardiner, Growth properties of $p$th means of potentials in the unit ball, Proc. Amer.
Math. Soc. {\bf 103} (1988) 861--869.

\bibitem{Gr68} A.F. Grishin, On growth regularity of subharmonic functions, II, Theory of functions, func. anal. and appl. (1968), no.7, 59--84 (in Russian)

\bibitem{Gr1} A.Grishin, Continuity and asymptotic continuity of
subharmonic  functions, Math. Physics, Analysis, Geometry, ILPTE {\bf 1} (1994),
no.2, 193--215 (in Russian).



\bibitem{FG} M.A. Fedorov, A.F.Grishin, Some questions of the Nevanlinna theory for
the complex half-plane, Math. Physics, Analysis and  Geometry (Kluwer Acad. Publish.) {\bf 1} (1998),
no.3, 223--271.

\bibitem{Ha} {W.K.Hayman,} Meromorphic functions, Oxford, Clarendon
press, 1964.



%
%
%
%



\bibitem{HK} W.K.Hayman, P.B.Kennedy, Subharmonic functions, V.1.
Academic press, London-New York-San Francisco, 1976.



\bibitem{L_rep} Linden C.N. The representation of regular functions, J. London Math. Soc. {\bf 39} (1964), 19--30.

\bibitem{Li71}
C.N.\ Linden, On Blaschke products of restricted growth, Pacific J. of Math., {\bf 38}  (1971) , no.2, 501--513.

\bibitem{Li_mp}
C.N.\ Linden, Integral logarithmic means for regular functions, Pacific J. of Math., {\bf 138}  (1989) , no.1, 119--127.

\bibitem{Li_mp1} C.N.\ Linden, {The characterization of orders for regular functions},
{ Math. Proc. Cambrodge Phil. Soc.} {{111}}  (1992), no.2, 299--307.


\bibitem{MR} G.R. MacLane, L.A. Rubel, On the growth of the Blaschke products, Canad. J. Math. {\bf 21}  (1969), 595--600.

\bibitem{VM} Mykytyuk, Ya. V., Vasyl’kiv, Ya. V., The boundedness criteria of integral means of Blaschke product logarithms, Dopov. Nats. Akad. Nauk Ukr., Mat. Prirodozn Tekh. Nauki {\bf 8} (2000), 10--14. (in Ukrainian)


\bibitem{Ry} A.V.Rybkin, Convergence of arguments of Blaschke
products in $L_p$-metrics, Proc. Amer. Math. Soc. {\bf 111}
(1991), no.3, 701--708.

\bibitem{stein} Stein E.M. Singular integrals and differentiability  properties  of functions, Princeton University Press, Princeton, New Jersey, 1970.

\bibitem{Sto83} Stoll M. On the rate of growth of the means $M_p$ of holomorphic and pluriharmonic functions on the ball, J.Math. Anal. Appl. {\bf 93} (1983), 109--127.

\bibitem{Sto89}   Stoll M., Rate of growth of $p$th means of
invariant potentials in the unit ball of $\mathbb{C}^n$, J.\ Math. Anal. Appl. {\bf 143} (1989), 480--499.

\bibitem{Tsuji}  Tsuji M. {Potential theory in modern function theory}. -- Chelsea Publishing Co.
                 Reprinting of the 1959 edition. New York, 1975.



%
%
%
%


%
%
%
%
%
%
%
%
%
%
%


\end{thebibliography}
\end{document}